\documentclass[graybox]{svmult}
\makeatletter
\providecommand*{\toclevel@title}{0}
\providecommand*{\toclevel@author}{0}
\makeatother

\usepackage{mathptmx}       
\usepackage{helvet}         
\usepackage{courier}        
\usepackage{type1cm}        
\usepackage{makeidx}         
\usepackage{graphicx}        
\usepackage{multicol}        
\usepackage[bottom]{footmisc}

\usepackage{amsmath, amsfonts, amssymb}
\usepackage{hyperref}

\usepackage[backend=biber, style=trad-abbrv, url=false, eprint=false, sortcites=true, doi=false]{biblatex}
\AtBeginBibliography{\small}
\usepackage{xcolor}
\usepackage{tabularx}

\makeindex             

\usepackage{mathtools}
\mathtoolsset{showonlyrefs}
\newtheorem{assumption}{Assumption}

\renewenvironment{proof}[1][Proof]
{%
  \par\noindent\textit{#1.}\quad
}
{%
  \nobreak\hfill\(\square\)\par
}

\begin{document}

\title*{Finite-Data Error Bounds for Approximating the Koopman Operator: Sampling Measures, Super-Polynomial Convergence and Regularization}
\titlerunning{Finite-Data Error Bounds for Approximating the Koopman Operator} 
\author{Daniel Fassler, Rachel Morris, Jason Bramburger and Simone Brugiapaglia}
\institute{Daniel Fassler \at Concordia University, 1455 De Maisonneuve Blvd. W.
Montreal, QC  H3G 1M8, \email{daniel.fassler@mail.concordia.ca}
\and Rachel Morris \at 
 \email{rachel.morris@mail.concordia.ca}
 \and Jason Bramburger \at 
 \email{jason.bramburger@concordia.ca}
  \and Simone Brugiapaglia \at 
 \email{simone.brugiapaglia@concordia.ca}}
%
%
\maketitle

\abstract{The Koopman operator is a well-established framework for lifting nonlinear dynamical systems to an infinite-dimensional space where dynamics are linear. Extended Dynamic Mode Decomposition (EDMD) is a widely used method in data-driven dynamics as it provides a Galerkin approximation of the Koopman operator on the finite-dimensional span of a prescribed dictionary of observable functions. In particular, EDMD only requires samples from the dynamical system, describing the Koopman operator without knowledge of the underlying dynamics. While many studies analyze asymptotic convergence results for EDMD in terms of data, the question of more practical finite-data convergence results remains incomplete in general settings. In this work, we prove bounds for the finite-data EDMD Galerkin approximation of the Koopman operator by a rate proportional to the inverse of the root of the number of samples (the Monte Carlo rate). These results apply to multiple classes of problems, i.e., for discrete or continuous dynamics, including stochastic systems. In certain more structured setting, we demonstrate super-polynomial rates are achievable both theoretically and computationally. Finally, we derive recovery guarantees for an EDMD variant in the undersampled regime inspired by compressed sensing techniques.}

\section{Introduction}
Nonlinear dynamical systems arise naturally throughout problems in science and engineering. The \textit{Koopman operator} provides an exact linear representation of nonlinear dynamics by lifting the temporal evolution of the state space $\mathbb{X}$ to a function space. Specifically, the Koopman operator acts on \textit{observable functions} $\varphi:\mathbb{X} \to \mathbb{C}$ and evolves states $\mathbf{x}_t$ in time within $\varphi$. It is defined as $\mathcal{K}^\tau[\varphi](\mathbf{x}) = \mathbb{E}\left[\varphi(\mathbf{x}_{t+\tau}) | \mathbf{x}_t = \mathbf{x}\right]$, with $\tau$ controlling the time evolution in $\varphi$. The Koopman formalism and recent advances in sensing and computation have led to faster, more precise methods for gathering dynamical data has placed data-driven methods at the forefront of modern dynamical system analysis~\cite{annurev:/content/journals/10.1146/annurev-fluid-011212-140652,dmd_theory_applications, TranChaos, SchaefferSparseRecovery}.

Extended Dynamic Mode Decomposition (EDMD)~\cite{williams2015data} has emerged at the forefront of methods for approximating the Koopman operator from data by through a least-squares fit on a user-specified dictionary of observables. EDMD provides a bridge between theory and application, as it is a feasible method of approximating the operator on a set of observables from measurements alone. Asymptotic convergence results exist~\cite{bramburger_auxiliary_2023, korda2018convergence}. Moreover, explicit finite-data rates are also available for particular systems or dictionaries~\cite{yadav2025approximationkoopmanoperatorbernstein,zhang_zuazua23} and in more general settings~\cite{philipp2026variance,nuske_finite-data_2022}. However, in all cases, the convergence of the error obeys the classical Monte Carlo rate $n^{-1/2}$ for $n$ samples for square EDMD.\\ 

\noindent \textbf{Main Contributions: }The goal of this paper is to present unified convergence results for EDMD in the finite data regime by creating a direct link between the method and the classical theory of least-squares approximation for identically, independently distributed (i.i.d.) data. We first recover the Monte Carlo rate in Theorem~\ref{thm:EDMD_MonteCarlo_Convergence} using a proof technique that is agnostic to the type of system (deterministic/stochastic and discrete/continuous). Theorem~\ref{thm:EDMD_better_rate} leverages recent advances~\cite{adcock_monte_2023} in least-squares approximation to obtain novel super-polynomial convergence rates by letting the dictionary size grow with the number of data points in regimes with sufficient regularity. Finally, in Theorem~\ref{thm:recovery_EDMD_LASSO} and~\ref{thm:recovery_EDMD_QCBP}, we present recovery guarantees in the undersampled regime for a regularized version of EDMD by applying compressed sensing results. Each theorem is backed by numerical simulations, experiments are reproducible and available at {\small \url{https://github.com/Falsorr/Finite-Data-Error-bounds-for-EDMD}}.

\section{Convergence rates for the Monte Carlo approximation}\label{sec:MC_approx}

In this section, we present the main convergence result of the EDMD algorithm for i.i.d.\ sampling. We extend the numerical and theoretical convergence results obtained in~\cite{ nuske_finite-data_2022, doi:10.1137/23M1597873, yadav2025approximationkoopmanoperatorbernstein}, unifying results in the stochastic and deterministic settings. 

Let $\mathbb{X} \subseteq \mathbb{R}^d$ be the state space with a probability measure $\rho$ and the points $(\mathbf{x}_t^{(i)}, \mathbf{y}_t^{(i)})$ with $\mathbf{x}_t^{(i)}\stackrel{\text{i.i.d.}}{\sim}\rho$ and $\mathbf{y}_t^{(i)} = \mathbf{x}_{t+\tau}^{(i)}, \tau > 0$. Consider $\vec{\psi} = (\psi_1, \ldots, \psi_m)$ the expansion dictionary and $\vec{\phi} = (\phi_1, \ldots, \phi_l)$ the feature dictionary with $\psi_i, \phi_j \in L^2_\rho(\mathbb{X})$, both of which can be complex valued. Define the EDMD matrices 
   $ \Psi_n~=~(\psi_i(\mathbf{x}_t^{(j)}))_{(i,j)=1}^{m,n}$ and $\Phi_n^\tau~=(\phi_i(\mathbf{y}_t^{(j)}))_{(i,j)=1}^{l,n}.$
EDMD computes a finite-dimensional approximation of the Koopman operator $K_{mn}^\tau$, a solution of the least-squares problem:
\begin{equation}\label{eq:edmd_minimization}
    \min_{K \in \mathbb{C}^{l\times m}} \|\boldsymbol{\Phi_{n}^\tau} - K\boldsymbol{\Psi_n}\|_F^2.
\end{equation}
The EDMD approximation of the Koopman operator is $\mathcal{K}_{mn}^\tau\varphi := \mathbf{c}\cdot K_{mn}^\tau\vec{\psi}$ for a function $\varphi = \mathbf{c}\cdot\vec{\phi}$, $\mathbf{c} \in \mathbb{C}^l$. Define $K_m^\tau \coloneqq \lim_{n\to\infty} K_{mn}^\tau$, which is the matrix representation of projection of the Koopman operator on the span of our expansion dictionary $\mathcal{K}_{m}^\tau\varphi \coloneqq \mathbf{c}\cdot K_m^\tau \cdot \vec{\psi}$.

Typically, EDMD uses $\vec\psi = \vec\phi$, but different dictionaries can provide a more accurate Galerkin approximation of the Koopman Operator and its generator, the Lie Derivative. This means that EDMD can be used to understand global properties of dynamical systems~\cite{bramburger_auxiliary_2023, bramburger_invariant_2024}. Additionally, it also implies our results and can be used to extend our results to other techniques by framing them as EDMD methods. For example, model identification can be cast as an EDMD problem by taking $\vec\phi = (\mathbf{x} \mapsto x_i)_{i=1}^d$ and $\vec\psi$ as a user-specified dictionary. 
\begin{theorem}[EDMD with i.i.d.\ samples converge at the rate $\mathcal{O}\left(n^{-1/2}\right)$]\label{thm:EDMD_MonteCarlo_Convergence}
    Let $\tau > 0$, $0 < \delta,\epsilon<1$. Moreover, suppose $\vec\psi$ are orthonormal functions in $L^2_\rho(\mathbb{X})$, $\sum_{i=1}^m\psi_i^2 \in L^2_\rho(\mathbb{X})$ and  that $\|\mathcal{K}_m^\tau [\phi_j]\|_{L^2_\rho(\mathbb{X})} \leq L_j$ for some $L_j > 0$. Then there exists $c_\delta > 0$, $\kappa = \|\sum_{i=1}^m\psi_i^2\|_{L^\infty(\mathbb{X})}$ such that for all $n \geq c_\delta\cdot\kappa\cdot\log(\frac{m}{\epsilon})$ the following holds: Let $\{\mathbf{x}_t^{(i)}\}_{i = 1}^n $ be $n$ data points sampled i.i.d.\ from $\rho$ and $\mathbf{y}_t^{(i)} = \mathbf{x}_{t+\tau}^{(i)}$, and let $K_{mn}^\tau$ be the solution to~\eqref{eq:edmd_minimization}. Define the diagonal matrix $A$ by $A_{jj} = \min\left\{1, \frac{L_j}{\|\mathcal{K}_{mn}^\tau[\phi_j]\|_{L^2_\rho(\mathbb{X})}}\right\}$ and  $\widetilde{K}_{mn}^\tau = AK_{mn}^\tau$. Then, with probability at least $1-l\epsilon$,
\begin{equation}\label{eq:monte_carlo_rate_EDMD}
        \resizebox{\textwidth}{!}{
        $\displaystyle\mathbb{E}\|\widetilde{K}_{mn}^\tau - K_{m}^\tau\|_F^2 \leq \frac{1}{n}\cdot\frac{\left\|\sum_{i=1}^m \psi_i^ 2\right\|_{L^2_\rho(\mathbb{X})}}{(1-\delta)^2}\sum_{j=1}^l\left\|\left(\mathcal{K}^\tau[\phi_j] - \mathcal{K}_m^\tau[\phi_j]\right)^ 2\right\|_{L^2_\rho(\mathbb{X})} + 4 \left(\sum_{j=1}^l L_j^2\right)\epsilon.$}
    \end{equation}
\end{theorem}

\begin{proof}
     We present here a sketch of the proof (see \cite[Theorem 3.2]{fassler_thesis_2025} for full details). 
    
    \textbf{Step 1: Decompose into $l$ least squares problems.} By standard norm properties, 
    \begin{equation}
        \min_{K \in \mathbb{C}^{l\times m}} \|\boldsymbol{\Phi_{n}^\tau} - K\boldsymbol{\Psi_n}\|_F^2 = \sum_{i=1}^l \min_{\mathbf{k}_i \in \mathbb{C}^m} \|\mathbf{b}_i - \Psi_n^*\mathbf{k}_i\|_2^2,
    \end{equation}
    where $\mathbf{b}_i, \mathbf{k}_i$ are the $i$th columns of $(\Phi_n^\tau)^*$ and $K^*$, respectively.
    
    \textbf{Step 2: Convergence rate for each problem.}  We  introduce the truncation operators $\mathcal{F}_{j}$ for $j\in[l]$ that project $\phi_j$ onto the $L^2_\rho(\mathbb{X})$-ball of radius $L_j$. With probability at least $1-\epsilon$, we obtain the following rate using results from~\cite[Chapter 5]{adcock_book_2022}:
    \begin{equation}
\resizebox{\textwidth}{!}{$
    \displaystyle\mathbb{E}\ \bigl\| \mathcal{F}_{j}(\mathcal{K}_{mn}^{\tau}[\phi_j]) - \mathcal{K}_m^{\tau}[\phi_j] \bigr\|_{L^2_\rho(\mathbb{X})}^{2}\leq\frac{\left\|\sum_{i=1}^m \psi_i^2\right\|_{L^2_\rho(\mathbb{X})}}{n(1-\delta)^2} \left\|\left(\mathcal{K}^\tau[\phi_j] - \mathcal{K}_m^\tau[\phi_j]\right)^2\right\|_{L^2_\rho(\mathbb{X})} + 4L_j^2\epsilon . $}
    \end{equation}
    
    \textbf{Step 3: Recombine and determine the convergence rate.} Due to the truncation step, we obtain an estimate between $\widetilde{K}_{mn}^\tau = AK_{mn}^\tau$ and $K_m^\tau$, where $A$ is the matrix representation of the truncation. All bounds hold with probability $1- l\epsilon$. 
    
\end{proof}

Theorem~\ref{thm:EDMD_MonteCarlo_Convergence} establishes Monte Carlo convergence for EDMD with minimal assumptions. The prefactor in the error bound depends on the choice of the expansion dictionary. One should aim for a small but expressive dictionary to keep both $\left\|\sum_{i=1}^m \psi_i^2\right\|_{L^2_\rho(\mathbb{X})}$ and $\|(\mathcal{K}^\tau[\phi_j] - \mathcal{K}_m^\tau[\phi_j])^2\|_{L^2_\rho(\mathbb{X})}$ small. The $L_j$ are an a priori upper bound on $\|\mathcal{K}_m^\tau[\phi_j]\|_{L^2_\rho(\mathbb{X})}$. In practice, we might not be able to compute the $L_j$ efficiently; however, this is not restrictive as $\epsilon$ can be chosen so that $\epsilon \leq m \exp(-\frac{n}{c_\delta \kappa})$. The term $(\sum_{j=1}^l L_j^2)\epsilon$ decays exponentially, hence the dominating factor in the error bound is decreasing at rate $\mathcal{O}(n^{-1/2})$.

Figure~\ref{fig:monte_carlo_experiments} contains the numerical experiments to validate the rate found in Theorem~\ref{thm:EDMD_MonteCarlo_Convergence}. We study the convergence of EDMD for different systems: the logistic map $x_{n+1} = \mu_nx_n^2 - 1$ with data drawn from the interval $[-1, 1]$, both in the chaotic deterministic ($\mu_n = 2)$ and stochastic ($\mu_n\sim \text{Unif}(0,2)$) regimes, and the Thomas model,
\begin{equation}
    \dot x_1 = 0.2\sin(5x_2) - bx_1, \qquad \dot x_2 = 0.2\sin(5x_3) - bx_2, \qquad\dot x_3 = 0.2\sin(5x_1) - bx_3,
\end{equation}
rescaled so that the state space is the unit cube for $ b = 0.208186$. These three models illustrate the generality of Theorem~\ref{thm:EDMD_MonteCarlo_Convergence}, across deterministic and stochastic, discrete and continuous systems. We consider $\vec{\psi} = \vec{\phi}$ and compare the performance of the Legendre and Chebyshev polynomials, monomials, and the Fourier basis.

\begin{figure}[t]
    \centering
    \includegraphics[width=\textwidth]{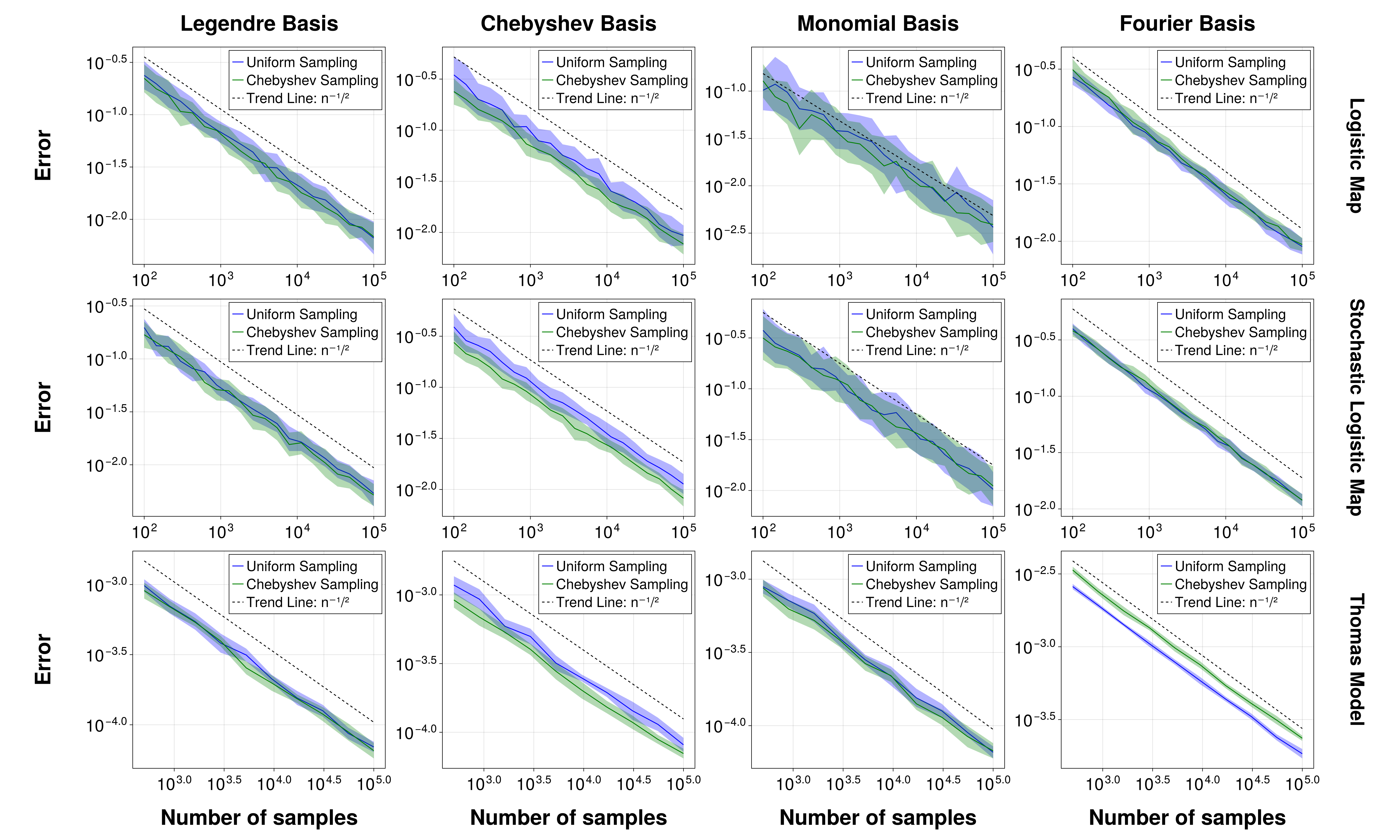}
    \caption{Convergence of $K_{mn}^\tau$ to $K_m^{\tau, HF}$ for fixed dictionary size}
    \label{fig:monte_carlo_experiments}
\end{figure}

The experiments are designed as follows. We know from~\cite{bramburger_auxiliary_2023} that $ \lim_{n\to\infty} K_{mn}^\tau = K_m^\tau$, so we measure the error by computing a ``high-fidelity" approximation of the Koopman matrix $K_m^{\tau, HF}$ through an EDMD approximation with $10^7$ samples. Then, we applied EDMD with a sample size varying from $n = 10^2$ to $n = 10^5$ and computed the relative error $\|K_{mn}^\tau - K_m^{\tau, HF}\|_F / \|K_m^{\tau, HF}\|_F$ with $10$ repetitions.

In all cases, the error decays at the Monte Carlo rate, even in cases where $\vec{\psi}$ is not necessarily an orthonormal set in $L^2_\rho(\mathbb{X})$. This shows linear independence is sufficient for the rate of Theorem~\ref{thm:EDMD_MonteCarlo_Convergence} (one could always obtain an orthonormal dictionary from a linearly independent basis through Gram-Schmidt). Numerical performance is, however, impacted by using a linearly independent but not orthonormal set. The matrices $\Psi_n$ and $\Phi_n$ become ill-conditioned for a linearly independent dictionary as $n$ increases. For each system, the choice of dictionary and sampling measure impacts only the prefactor, making little difference in most cases except for the Thomas model with the Fourier basis exhibiting cleaner separation, with uniform sampling outperforming Chebyshev sampling.
Better results can be achieved in setups where $\vec\phi \subset \vec\psi$. A further investigation is done in~\cite[Section 3.2]{fassler_thesis_2025}.

\section{Speeding up the Monte Carlo rate}\label{sec:exp_rate}
In this section, we present improved EDMD convergence rates when the dynamics and expansion dictionary extend holomorphically to a Bernstein polyellipse, using Legendre or Chebyshev dictionaries with uniform or Chebyshev sampling measures. The \textit{Bernestein ellipse} for parameter $\rho > 1$  is $\mathcal{E}_\rho = \left\{(z+z^{-1})/2 : z\in\mathbb{C}, 1 \leq |z| \leq \rho \right\}$.
    Given $\vec{\rho} = (\rho_1, \rho_2, \ldots, \rho_N)$, the \textit{Bernstein polyellipse} for parameter $\vec{\rho}$ is $\mathcal{E}_{\vec{\rho}} = \mathcal{E}_{\rho_1} \times \mathcal{E}_{\rho_2} \times \ldots \times \mathcal{E}_{\rho_N} \subset \mathbb{C}^N$. Using results in~\cite{adcock_monte_2023} we show faster convergence is achieved by enlarging the expansion dictionary as $n$ increases.
\begin{theorem}[Faster-than-polynomial convergence rate for EDMD.]\label{thm:EDMD_better_rate}
    Let $0 < \epsilon, p < 1, \ \rho$ be either the uniform or Chebyshev measure on $\mathbb{X} = [-1, 1]^N$, and $\left\{\vec{\psi}_\nu\right\}_{\nu \in \mathbb{N}_0^N}$ be the tensorized Legendre or Chebyshev polynomials, respectively. For $n\geq 3$ take $\{\mathbf{x}_t^{(i)}\}_{i = 1}^n $ be $n$ data points sampled i.i.d.\ from $\rho$ and $\mathbf{y}_t^{(i)} = \mathbf{x}_{t+\tau}^{(i)}$. Then there exist sets $S_i \subseteq \mathbb{N}_0^N$ of cardinality $|S_i| \leq \lceil n / \log(n/\epsilon)\rceil$ such that for all $i \in [l]$, we have that with probability at least $1-l\epsilon$, the following holds. If $\mathcal{K}[\phi_i]$ admits a holomorphic extension to a Bernstein polyellipse $\mathcal{E}_{\vec{\rho}} \subset \mathbb{C}^N$ for all $i \in [l]$, then the EDMD approximation $\mathcal{K}_{mn}^\tau$ of the Koopman operator  $\mathcal{K}^\tau$ with expansion dictionary $\vec{\psi} = \left\{\vec{\psi}_\nu\right\}_{\nu \in \cup_i S_i} $ is unique and satisfies the following bound:\begin{equation}\label{eq:EDMD_better_rate}
        \|\mathcal{K}_{mn}^\tau[\phi_i] - \mathcal{K}^\tau[\phi_i]\|_{L^2_\rho(\mathbb{X})} \leq C(\vec{\rho}, p) \left(\frac{n}{\log(n/\epsilon)}\right)^{\frac{1}{2} - \frac{1}{p}},
    \end{equation}
    where $C(\vec{\rho}, p)$ is a constant that depends on $\vec{\rho}$ and $p$ only.
\end{theorem}

\begin{proof}
    We write the EDMD problem as $l$ independent least-squares problems and then bound each of the problems separately using~\cite[Theorem B.2]{dinverno2025surrogatemodelsdiffusiongraphs}.   
\end{proof}

Although Theorem~\ref{thm:EDMD_better_rate} is an existence result, its assumptions hold for many systems, including polynomial systems and those involving entire functions. A key feature of this result is that the expansion dictionary grows with the number of data points. The bound provided in the statement of the theorem holds for any $0 < p < 1$, showing that this convergence is super-polynomial. 

Figure~\ref{fig:exponential_rate} demonstrates super-polynomial convergence numerically for the deterministic logistic map with $\vec{\psi}=\vec{\phi}$, $|\vec{\psi}|=m + 1$ where $m$ is the maximum degree of the polynomials in the basis, and $n=\lceil3(m+1)^2\log(m+1)\rceil$. Multiple polylogs scaling were tested, and this one performed the best empirically. We replace the high-fidelity Koopman matrix comparison of Section~\ref{sec:MC_approx} with a Monte Carlo approximation of the relative $L^2$-error, using $M=10n$ i.i.d. samples $\{\mathbf{z}_i\}_{i=1}^M$ from $\rho$ and the test observable $\varphi(x)=e^x$. Monomials become unstable near $n=4000$ due to the ill-conditioning of $\Psi_n$, unlike the orthogonal Legendre and Chebyshev bases.

\begin{figure}[t]
    \centering
\includegraphics[width=0.7\linewidth]{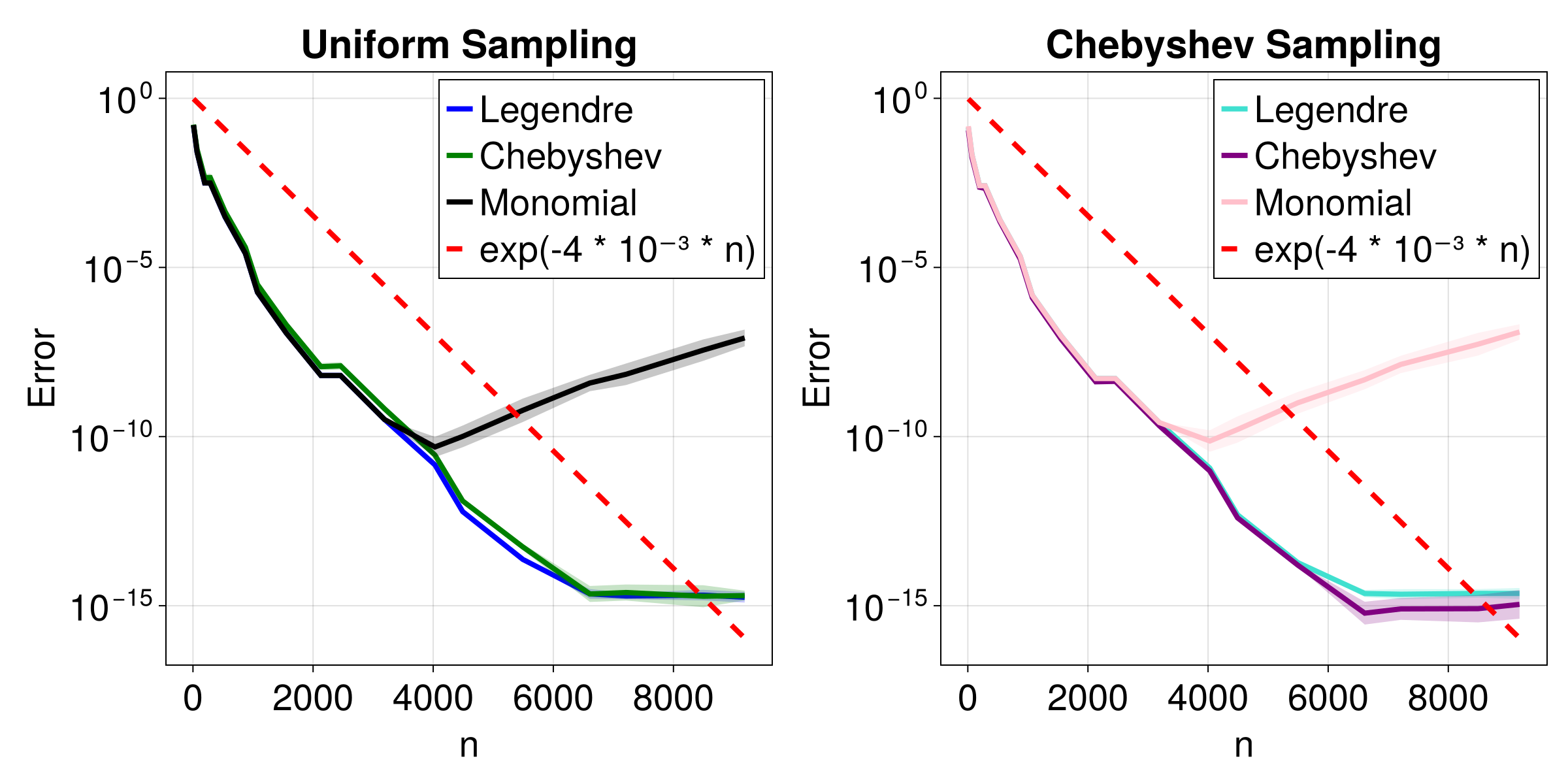}
    \caption{Convergence of $\mathcal{K}_{mn}^\tau$ to $\mathcal{K}^\tau$ for varying sample dictionary size and, $n = \lceil3(m+1)^2\log(m+1)\rceil$ for the logistic map in chaotic regime.}
    \label{fig:exponential_rate}
\end{figure}

\section{Recovery in the undersampled regime}

In this section, we present a preliminary investigation of compressed sensing methods for approximating the Koopman operator in the undersampled regime. We state our recovery guarantees in terms of the best $s$-term approximation error. The $\ell_1$ best $s$-term approximation error of a vector $\mathbf{c}\in\mathbb{C}^n$ is 
\( \sigma_s(\mathbf{c}) = \min_{\|\mathbf{z}\|_0 \leq s}\|\mathbf{c} - \mathbf{z}\|_1\)
where $\|\mathbf{z}\|_0$ counts the number of nonzero entries of $\mathbf{z}$. To obtain the desired recovery guarantees, we require stronger assumptions on our dictionary.
\begin{assumption}\label{as2}
    Suppose $\vec\psi$ are orthonormal functions in $L^2_\rho(\mathbb{X})$; moreover, assume there exists a constant $B \ge 1$ such that $\|\psi_j\|_\infty \leq B$ for all $j\in[m]$.
\end{assumption}

Assumption~\ref{as2} ensures $\vec\psi$ is a Bounded Orthonormal System (BOS). Recovery results in compressed sensing  rely on the measurement matrix having the Restricted Isometry Property (RIP) and the robust Null Space Property (rNSP). Sampling from a BOS yields both of the properties with high probability~\cite{adcock_book_2022}. All dictionaries considered in this paper are BOS except the monomials. We study the LASSO and QCBP variants of EDMD, inspired by the results of~\cite{SchaefferSparseRecovery}.
    We define the LASSO-EDMD and QCBP-EDMD problems as 
    \begin{align}
         &\min_{K\in \mathbb{C}^{l\times m}} \|{\Phi}_n^\tau - K\Psi_n\|_F^2 + \lambda\|K\|_{F,1}, \label{eq:LASSO_EDMD} \\
         &\min_{K\in \mathbb{C}^{l\times m}} \|K\|_F ~~\mathrm{ subject}~\mathrm{to }~~ \|\Phi_n^\tau - K\Psi_n\|_F \leq \sigma, \label{eq:BP_EDMD}
    \end{align}
    where $\|A\|_{F,1} = \sum_{i=1}^{l}\sum_{j=1}^{m}|A_{ij}|$ is the entrywise $\ell_1$-norm.
In comparison to classical EDMD \eqref{eq:edmd_minimization}, LASSO-EDMD regularization of the matrix approximation of the Koopman operator $K$. This can be understood as a ridge regression problem that discourages an overfitting of the data by instead seeking to recover  the dominant behaviour of the system. We use an $\ell_1$ penalization instead of an $\ell_2$ penalization to promote sparse features in our approximation. 

\begin{theorem}[Recovery guarantee for LASSO-EDMD]\label{thm:recovery_EDMD_LASSO}
    Let $\tau > 0$, $s \in [m],  0 < \epsilon < 1, 0 < \delta \leq \sqrt{2} - 1$, and suppose Assumption~\ref{as2} holds. There exists a universal $C > 0$ and, where $L(s) = \log\left(2B^2s\right)\left[\log\left(2B^2s \right)\cdot\log(2m)+\log\left(2\varepsilon^{-1}\log\left(2B^2s\right)\right)\right]$,
    so that for all $n\geq C\cdot \delta^{-2}\cdot B^2\cdot s\cdot L(s)$, the following holds: Let $\{\mathbf{x}_t^{(i)}\}_{i = 1}^n $ be $n$ data points sampled i.i.d.\ from $\rho$ and $\mathbf{y}_t^{(i)} = \mathbf{x}_{t+\tau}^{(i)}$. Suppose $K_m^\tau \in \mathbb{C}^{l \times m }$ satisfies $\Phi_n^\tau = K_m^\tau \Psi_n^\tau + E$ where $E \in \mathbb{C}^{l\times n}$ is a matrix of errors with $\|E\|_F \leq \eta$ for some $\eta > 0$. Take $\lambda \in (c_1 \eta/\sqrt{s}, c_2 \eta/\sqrt{s})$ for $c_2 \ge c_1>0$. Let $\left(K_{mn}^\tau\right)^\#$ be the solution of LASSO-EDMD~\eqref{eq:LASSO_EDMD}. Then with probability greater than $1-\epsilon$, there exist constants $C_1,D_1>0$ such that
\begin{equation}\label{eq:LASSO_EDMD_bounds}
    \bigl\|(K_{mn}^{\tau})^\# - K_m^\tau\bigr\|_{F}^{2}
    \leq
    \sum_{i=1}^{l}
    \left[
    C_1 \frac{\sigma_s(\mathbf{k}_{m,i})}{\sqrt{s}}
    + D_1\eta
    \right]^2 .
\end{equation}
where $\mathbf{k}_{m,i}$ is the $i$th column of $K_m^\tau$.
\end{theorem}
\begin{proof}
   By~\cite[Section 4.1.2]{fassler_thesis_2025}, the LASSO-EDMD problem~\eqref{eq:LASSO_EDMD} can be written as
    \begin{equation}\label{eq:classical_LASSO_for_EDMD}
        \min_{K \in \mathbb{C}^{l\times m}} \|\Phi_n^\tau - K\Psi_n\|_F^2 + \lambda\|K\|_{F,1} = \sum_{i=1}^l\min_{\mathbf{k}_i \in \mathbb{C}^m} \|\mathbf{b}_i - \Psi_n^*\mathbf{k}_i\|_2^2 + \lambda\|\mathbf{k}_i\|_1
    \end{equation}
    where $\mathbf{k}_i, \mathbf{b}_i$ are the columns of $K^*$ and $(\Phi_n^\tau)^*$, respectively. Since $\Psi_n^*$ is a measurement matrix from sampling a BOS, we apply~\cite[Theorem 6.15]{adcock_book_2022} to conclude that with probability greater than $1-\epsilon$, $\Psi_n^*$ has the RIP of order $2s$ with $\delta_{2s} < \delta \leq  \sqrt{2} - 1$. Next we apply~\cite[Theorem 6.11]{adcock_book_2022} to conclude that $\Psi_n^*$ has the rNSP of order $s$ with constants $\rho, \gamma$ that depend on $\delta_{2s}$ only. Finally, we use~\cite[Theorem 6.28]{adcock_book_2022} to conclude that there exists constants $C_2, C_3$ depending on $\rho$ and $\gamma$ only such that
    \begin{equation}
        \|\mathbf{k}_i^\# - \mathbf{k}_{m,i}\|_2 \leq C_2 \frac{\sigma_s(\mathbf{k}_{m,i})}{\sqrt{s}} + \frac{\lambda C_3^2}{8C_2} \sqrt{s} + \frac{C_2}{2}\frac{\|\mathbf{e}_i\|_2^2}{\lambda\sqrt{s}} + \frac{C_3}{2}\|\mathbf{e}_i\|_2,
    \end{equation}
    where $\mathbf{k}_i^\#$ is the solution of~\eqref{eq:classical_LASSO_for_EDMD} and $\mathbf{e}_i$ is the $i$th column of $E^*$. Since we have $c_1\frac{\eta}{\sqrt{s}} \leq \lambda \leq c_2 \frac{\eta}{\sqrt{s}}$, the bounds can be simplified to $\|\mathbf{k}_i^\# - \mathbf{k}_{m,i}\|_2 \leq C_1 \frac{\sigma_s(\mathbf{k}_{m,i})}{\sqrt{s}} + D_1\eta$,
     where $D_1$ is a constant depending on $\rho$ and $\gamma$ only. We recover~\eqref{eq:LASSO_EDMD_bounds} by noticing that $\mathbf{k}_i^\#$ is the $i$th row of $\left(K_{mn}^\tau\right)^\#$ and using standard properties of the Frobenius norm. 
\end{proof}

Theorem~\ref{thm:recovery_EDMD_LASSO} shows that with sufficient data, the error in the LASSO-EDMD approximation of the Koopman operator is comparable to the best $s$-term approximation of $K_m^\tau$ up to noise. If the Koopman operator is $s$-sparse in the representation of the dictionaries $(\vec\phi, \vec\psi)$, then $\sigma_s(\mathbf{k}_{m,i}) = 0$, and we recover $K_{m}^\tau$ up to $\eta^2$.

The Quadratically Constrained Basis Pursuit (QCBP) problem~\eqref{eq:BP_EDMD} can also be cast as an EDMD problem for $\sigma = 0$. It does not suffer from a trade-off between the data-fit term and the regularizer like LASSO-EDMD but is less robust to noise.

\begin{theorem}\label{thm:recovery_EDMD_QCBP}
    Under the assumptions of Theorem~\ref{thm:recovery_EDMD_LASSO}, for all $n \geq C\cdot \delta^{-2}\cdot B^2\cdot s\cdot L(s)$,  the following holds: Let $\{\mathbf{x}_t^{(i)}\}_{i = 1}^n $ be $n$ data points sampled i.i.d.\ from $\rho$ and $\mathbf{y}_t^{(i)} = \mathbf{x}_{t+\tau}^{(i)}$. Suppose $K_m^\tau \in \mathbb{C}^{l \times m }$ satisfies $\Phi_n^\tau = K_m^\tau \Psi_n^\tau$. Let $\left(K_{mn}^\tau\right)^\#$ be the solution of LASSO-QCBP~\eqref{eq:BP_EDMD} with $\sigma = 0$. Then with probability greater than $1-\epsilon$, there exist a constant $C_1$ such that
    \begin{equation}
        \bigl\|(K_{mn}^{\tau})^\# - K_m^\tau\bigr\|_{F}^{2}
    \leq
    C_1\sum_{i=1}^{l}
     \frac{\sigma_s(\mathbf{k}_{m,i})^2}{s}.
    \end{equation}
\end{theorem}
\begin{proof}
    The proof is analogous to that of Theorem~\ref{thm:recovery_EDMD_LASSO} using~\cite[Theorem 6.9]{adcock_book_2022}. 
\end{proof}
Figure~\ref{fig:recovery_compressed_sensing} summarizes numerical experiments for LASSO-EDMD and QCBP-EDMD on the Lorenz-96 model~\cite{lorenz1996predictability} with $d=10$ and forcing constant $F=8$. Data are sampled along multiple trajectories, with $q$ points per trajectory, using degree-$2$ Legendre polynomials for $\vec{\psi}$ and degree-$1$ Legendre polynomials for $\vec{\phi}$. We vary the undersampling rate $n/(ml)$, i.e., the ratio of samples to unknowns in the optimization problem, and record the residuals on the validation set $\{\mathbf{z}_i\}_{i=1}^n$ also sampled along trajectories. For QCBP-EDMD, we use $\sigma > 0$ but small for numerical feasability. For both LASSO-EDMD the residual error hits a plateau below $10^{-1}$ while for QCBP-EDMD, it quickly reaches around $10^{-2}$ with an undersampling rate of $0.5$ or $420$ samples. Compared with EDMD, which required $10^5$ samples to achieve comparable accuracy for the Lorentz-96 model, exploiting the sparsity of the Koopman operator drastically reduces the sample complexity (experiment available at code repository).

\begin{figure}[t]
    \centering
        \includegraphics[width = \textwidth]{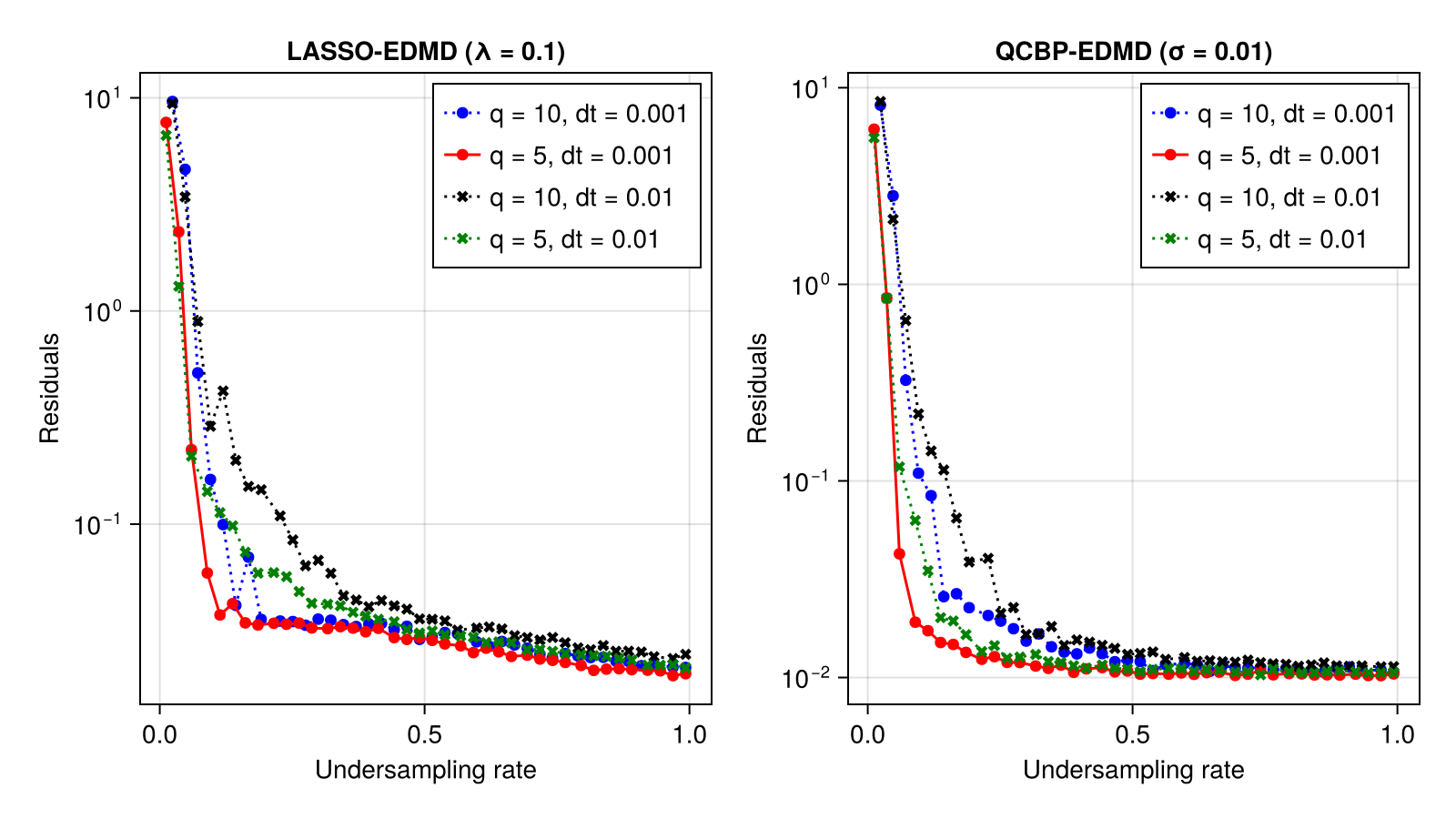}
    \caption{Recovery experiments for LASSO-EDMD and QCBP-EDMD.}\label{fig:recovery_compressed_sensing}
\end{figure}

\section{Discussion}

In this paper, we study EDMD as a least-squares approximation problem. We derive two convergence rates for i.i.d.\ sampling: a global classical $\mathcal{O}(n^{-1/2})$ rate under minimal assumptions and super-polynomial rates for Legendre and Chebyshev polynomial dictionaries under holomorphic dynamics. We consider the EDMD-LASSO and EDMD-QCBP methods, highlighting the benefits of sparsity by reducing the sample complexity required for accurate recovery in the undersampled regime using compressed sensing techniques. These results apply to related methods such as model identification by recasting them as an EDMD problem. The recovery guarantees of Theorems~\ref{thm:EDMD_MonteCarlo_Convergence} and~\ref{thm:EDMD_better_rate} could be extended to weighted EDMD with noisy measurements; we leave this for future work.

A separate challenge concerns the sampling assumption: while we restrict to i.i.d.\, sampling, dynamical data are naturally collected along trajectories, a setting in which EDMD may converge arbitrarily slowly~\cite{krengel1978speed}. For sampling along periodic or quasi-periodic trajectories, a complementary approach uses weighted Birkhoff averages~\cite{bou-sakr-el-tayar2026weighted} to accelerate convergence; this result does not work for chaotic systems like the logistic map where our super-polynomial convergence rates apply. A full investigation into which systems enjoy improved convergence rates for different sampling methods is one potential avenue of study. Another future direction is to extend Theorem~\ref{thm:EDMD_better_rate} to dictionaries other than Legendre and Chebyshev polynomials, for example to Jacobi polynomials or the Fourier basis on a periodic domains under assumption of Sobolev smoothness~\cite{adcock2026universalsampleoptimalalgorithmsrecovery}. While in this work we focus on EDMD and its variants, Koopman autoencoders are another popular method to approximate the Koopman operator that do not require a user-specified dictionary. Related work \cite{morris2026architectural} provides novel convergence guarantees in this setting. \\

\noindent \textbf{Acknowledgements. } DF acknowledges the help of Dr. Ben Adcock for enlightening discussions. SB acknowledges support from NSERC through grant RGPIN/2020-06766. All the authors acknowledge the support of FRQNT through grant 359708.


\printbibliography
\end{document}